\documentclass[a4paper,11pt]{article}
\usepackage[top=2.5cm,bottom=2.5cm,left=2.2cm,right=2.2cm]{geometry}
\usepackage{latexsym,amsmath,amssymb,amsfonts,epsfig,psfrag,url,graphics,ifpdf,multicol}
\usepackage{eepic,color,colordvi,amscd,amsthm,makecell,multirow,enumerate,enumitem}
\usepackage[section]{algorithm}
\usepackage{algpseudocode}
\usepackage[numbers,sort&compress]{natbib}
\usepackage{indentfirst,graphics,epsfig}
\usepackage{caption}
\usepackage{graphicx}
\usepackage{mathrsfs}
\usepackage{epstopdf}
\usepackage[T1]{fontenc}
\usepackage{subcaption}
\usepackage{newtxtext}
\usepackage{booktabs}
\usepackage{csquotes}

\usepackage{latexsym,bm}

\newtheorem{theo}{Theorem}[section]
\newtheorem{lem}[theo]{Lemma}
\newtheorem{coro}[theo]{Corollary}

\newtheorem{prop}[theo]{Proposition}

 \theoremstyle{definition}

\theoremstyle{remark}
\newtheorem{rema}[theo]{Remark}

\begin{document}
\title{An explicit basis for the extended Catalan arrangement of type $D_\ell$}

\author{Weili Guo$^{a}$ \footnote{Corresponding author: guowl@mail.buct.edu.cn (W. Guo).}, ~~Yupeng Wang$^{a}$, ~~Zixuan Wang$^{b}$\\[2mm]
\small $^{a}$College of Mathematics and Physics,
\small Beijing University of Chemical Technology, Beijing, 100029, P. R. China\\
\small $^{b}$Center for University Mathematics Teaching,
\small Nanjing Forestry University, Nanjing, 210037, P. R. China}
\date{}

\maketitle \vspace{0.25cm}

\noindent {\bf Abstract}: \small

Motivated by Kawanoue's construction with the finite-sum formulas, 
we construct an explicit homogeneous basis 
for the derivation module 
of the cone 
of the extended Catalan arrangement of type $D_{\ell}$.  
We further show that the restriction of this basis at infinity coincides, up to a nonzero constant, with the integral basis constructed by Feigin-Wang-Yoshinaga. 
This identifies the geometric motivation behind the parameter choices.

\vspace{2mm}


\vspace{6mm}

\baselineskip=0.25in

\section{Introduction}

\vspace{3mm}

Let $V$ be an $\ell$-dimensional vector space over the field $\mathbb{C}$, 
and $V^*$ be its dual space equipped with a basis $\left\{x_1,x_2,\cdots,x_\ell\right\}$. 
Let $S=\operatorname{Sym}(V^*)$ be the symmetric algebra of $V^*$. 
With respect to the chosen basis, 
we have the natural isomorphism $S\cong\mathbb{C}\left[x_1,x_2,\cdots,x_\ell\right]$. 
We denote by $\operatorname{Der}(S)$ the set of $\mathbb{C}$-linear maps $\theta:S\to S$ such that $\theta(fg)=f\theta(g)+g\theta(f)$ for all $f,g\in S$.  
The free $S$-module $\operatorname{Der}(S)$ has a basis $\partial x_1,\partial x_2,\dots,\partial x_\ell$. 
An element of $\operatorname{Der}(S)$ is called a derivation.

\vspace{1mm}

A hyperplane is an affine subspace of codimension one in $V$, 
written as  $H=\{v\in V \mid \alpha_H (v) = c \}$,
where $\alpha_H\in V^*\setminus\left\{0\right\}$ and $c\in \mathbb{C}$. 
A hyperplane arrangement $\mathcal{A}$ is a finite set of hyperplanes. 
We call $\mathcal{A}$ a central arrangement if all hyperplanes pass through the origin. 
For the central arrangement $\mathcal{A}$,
define the $S$-module $D(\mathcal{A})$ by
\begin{equation*}
	D(\mathcal{A})=\left\{\theta\in\operatorname{Der}(S)\mid
	\theta(\alpha_H)\in\alpha_HS,\text{~for any~}H\in\mathcal{A}\right\}.
\end{equation*}	
\noindent The arrangement $\mathcal{A}$ is called free if $D(\mathcal{A})$ is a free module. A nonzero element $\theta\in D(\mathcal{A})$ is homogeneous of polynomial degree $p$ if $\theta=\sum_{i=1}^{\ell}f_i\partial_{x_i}$, where each $f_i$ is a homogeneous polynomial of degree $p$ for $i=1,2,\dots,\ell$. In this case we write $\operatorname{pdeg}\theta=p$. For a free module $D(\mathcal{A})$ with a homogeneous basis $\left\{\theta_1,\theta_2,\dots,\theta_\ell\right\}$, we call $\operatorname{pdeg}\theta_1,\operatorname{pdeg}\theta_2,\dots,
\operatorname{pdeg}\theta_\ell$ the exponents of $\mathcal{A}$ and write 
\begin{equation*}
	\begin{split}
		\operatorname{exp}(\mathcal{A})=
		\left\{
		\operatorname{pdeg}\theta_1,\operatorname{pdeg}\theta_2,\dots,
		\operatorname{pdeg}\theta_\ell
		\right\}.
	\end{split}
\end{equation*}
For more details, see~\cite{P. Orlik-1992, G. Ziegler-1989}.
\vspace{1mm}

Let $\Phi$ be a crystallographic irreducible root system. 
Denote by $\Phi^+$ its positive system. 
For any $\alpha\in\Phi^+$ and $k\in\mathbb{Z}$, the extended Catalan arrangement can be defined by 
\begin{equation*}
	\operatorname{Cat}(\Phi,m)=
	\left\{ 
	\alpha-k=0
	\mid
	\alpha\in\Phi^+,
	-m\leq k\leq m
	\right\}.
\end{equation*}
The term \enquote{Catalan arrangement} arises from the result proved in~\cite{A. Postnikov-2000} that the number of chambers in the fundamental region of a type $A$ root system is equal to the Catalan number. After adjoining a coordinate $z$, denote by $\mathcal{C}_{\ell,m}=c\operatorname{Cat}(D_\ell,m)$ the cone of the extended Catalan arrangement of type $D_\ell$ for $m\geq0$, and we write its defining polynomial as
\begin{equation}\label{eq:Q}
	Q_{\ell,m}=z\prod_{1\leq i<j\leq\ell}
	\prod_{k=-m}^{m}(x_i-x_j-kz)(x_i+x_j-kz).
\end{equation}
Thus its defining hyperplanes are $z=0$ and $\alpha-kz=0$, where $\alpha\in\Phi^+(D_\ell)$ and $-m\leq k\leq m$.

There has been extensive research on Catalan arrangements. 
Edelman and Reiner \cite{P. H. Edelman-1996} proposed a conjecture about the freeness of the cone of the extended Catalan arrangement. 
Yoshinaga \cite{M. Yoshinaga-2004} used algebraic geometry methods to show its freeness. 
Recently, the explicit construction of bases attracted considerable attention from many researchers. 
Suyama and Yoshinaga \cite{D. Suyama-2021} constructed an explicit basis of the extended Catalan arrangement of type $A_{\ell-1}$ based on Bandlow-Musiker's integral formula.
Gao and Si \cite{R. Gao-2026} also provided a basis construction
by Bernoulli polynomial 
for this type.  
Feigin, Wang and Yoshinaga \cite{M. Feigin-2025} 
constructed a set of derivations 
in view of their integral expression of the basis of Coxeter arrangement of type $B_2$, 
and gave a conjecture about an explicit basis of the extended Catalan arrangement of type $B_2$. 
Kawanoue \cite{H. Kawanoue-2026} proved the conjecture and generalized the construction to the extended Catalan arrangement of type $B_\ell$.
The construction for type 
$D_\ell$ remains open.

\vspace{1mm}

In this paper, 
we fix $m$ and $\ell$.
Inspired by Kawanoue's work in~\cite{H. Kawanoue-2026}, we define the following polynomials and give an explicit basis of the cone of the extended Catalan arrangement of type $D_\ell$.

\vspace{1mm}

For an integer $a$, we define
\begin{equation}\label{eq:P}
	P(t,a)=
	\begin{cases}
		\displaystyle\prod_{q=-a}^{a}(t+q), &a\geq0,\\
		P(t,-a-1)^{-1}, &a<0.
	\end{cases}
\end{equation}
In particular $P(t,-1)=\frac{1}{t}$ and we assume empty products equal $1$.

\vspace{1mm}

For each $i$, we define the remaining variables in their natural order as $\mathbf{y}=(y_1,\dots,y_{\ell-1})=(x_1,\dots,\widehat{x_i},\dots,x_\ell)$. For $\mathbf{s}\in\left\{0,\dots,m\right\}^{\ell-1}$, write $S_j=s_1+\cdots+s_j$ and $S_0=0$. The regular coefficients, for $0\leq r\leq\ell-2$, are
\begin{equation}\label{eq:regular}
	p_{i,r}^{(m)}=\sum_{\mathbf{s}\in\left\{0,\dots,m\right\}^{\ell-1}}
	\frac{(-1)^{S_{\ell-1}}P(x_i,r+S_{\ell-1})}
	{\binom{r+S_{\ell-1}+\frac12}{m+1}}\cdot
	\prod_{j=1}^{\ell-1}\left[
	\binom{m}{s_j}\prod_{q=r-m+S_j+1}^{r+S_{j-1}}(y_j^2-q^2)\right].
\end{equation}
To write the exceptional coefficients without rational functions, define
\begin{equation*}
	R_N(t)=
	\begin{cases}
		1,&N=0,\\
		t^2\displaystyle\prod_{q=1}^{N-1}(t^2-q^2),&N\geq1.
	\end{cases}
\end{equation*}
Then write
\begin{equation}\label{eq:exceptional}
	p_{i,*}^{(m)}=\left(\prod_{j\ne i}x_j\right)
	\sum_{\mathbf{s}\in\left\{0,\dots,m\right\}^{\ell-1}}
	\frac{(-1)^{S_{\ell-1}}R_{S_{\ell-1}}(x_i)}
	{\binom{S_{\ell-1}-\frac12}{m+1}}\cdot
	\prod_{j=1}^{\ell-1}\left[
	\binom{m}{s_j}
	\prod_{q=-m+S_j}^{S_{j-1}-1}(y_j^2-q^2)\right].
\end{equation}
It is easy to check that every scalar denominator is nonzero, since its numerator is a product of nonzero half-integers. And these expressions are finite sums of polynomials, with $(m+1)^{\ell-1}$ summands per coefficient. 

\vspace{1mm}

Assume $d_r=2m(\ell-1)+2r+1$, $d_*=2m(\ell-1)+\ell-1$ and $X=\prod_{i=1}^{\ell}x_i$, we define the homogeneous polynomial by
\begin{equation*}
	f_{i,r}(x,z)=z^{d_r}p_{i,r}^{(m)}(\frac{x}{z}),\qquad
	f_{i,*}(x,z)=z^{d_*}p_{i,*}^{(m)}(\frac{x}{z}),
\end{equation*}
and we define the derivations by
\begin{equation}	
	\theta_r=\sum_{i=1}^{\ell}f_{i,r}\partial_{x_i},\qquad\label{eq:theta}
	\theta_*=\sum_{i=1}^{\ell}f_{i,*}\partial_{x_i},\qquad
	\theta_E=z\partial_{z}+\sum_{i=1}^{\ell}x_i\partial_{x_i}.
\end{equation}
\vspace{1mm}

With the above definitions, the main theorem of this paper is shown as follows.

\vspace{1mm}

\begin{theo}\label{thm:main}
	For $\ell\geq4$ and $m\geq0$, the ordered derivations
	\begin{equation*}
		(\theta_E,\theta_0,\theta_1,\dots,\theta_{\ell-2},\theta_*)
	\end{equation*}
	is a homogeneous basis of $D(\mathcal{C}_{\ell,m})$. In the coordinate order $(z,x_1,\dots,x_\ell)$, its coefficient determinant equals $C_{\ell,m}Q_{\ell,m}$, where
	\begin{equation}\label{eq:constant}
		C_{\ell,m}=
		\frac{\displaystyle(-1)^{m\frac{\ell(\ell-1)}{2}+\frac{(\ell-1)(\ell-2)}{2}}}
		{\displaystyle\binom{-\frac12}{m+1}
			\prod_{n=1}^{\ell-1}\binom{(m+1)n-\frac12}{m+1}}
		\ne0.
	\end{equation}
	And the exponents of $\mathcal{C}_{\ell,m}$ are  $\operatorname{exp}(\mathcal{C}_{\ell,m})=\left\{1,d_0,d_1,\dots,d_{\ell-2},d_*\right\}$.
\end{theo}
The assumption $\ell\geq4$ follows the usual convention for irreducible root systems of type $D$. The formulas and proofs remain valid for $\ell=2,3$, corresponding to $D_2\cong A_1\times A_1$ and $D_3\cong A_3$, respectively.

\vspace{1mm}

In this paper, we first investigate the properties of the coefficient polynomials associated with derivations in section 2. Then we compute the coefficient determinant and apply Saito's criterion to complete the proof of the basis theorem in section 3. Next we identify the restrictions of the constructed derivations at infinity with the integral basis of Feigin-Wang-Yoshinaga in section 4. Finally, we compute an explicit example for the case $\mathcal{C}_{4,1}$.

\vspace{3mm}

\section{Preliminaries}

\vspace{3mm}

\subsection{Finite-Sum Identities and Extension to Integer Parameters}

\vspace{3mm}

In this part, we introduce Kawanoue's finite sums and extend their identities to integer parameters.

\begin{lem}\label{lem:ratio}
	Let $a,b\in\mathbb{Z}$ and $a\geq b$, then we have
	\begin{equation*}
		\frac{P(t,a)}{P(t,b)}=\prod_{q=b+1}^{a}(t^2-q^2),
	\end{equation*}
	and $P(-t,a)=-P(t,a)$ for every integer $a$.
\end{lem}

\vspace{1mm}

\begin{proof}
	According to the definition~\eqref{eq:P}, we have $\frac{P(t,c)}{P(t,c-1)}=t^2-c^2$ for every integer $c$, including $c=0$ and negative $c$. Multiplying the recurrence over $c=b+1,\dots,a$ and cancelling the intermediate factors, we can get the desired identity. For nonnegative indices, oddness follows directly from the product formula. For negative indices, it follows from the reciprocal definition.
\end{proof}

\vspace{1mm}

For convenience, we adopt Kawanoue's finite sums in~\cite{H. Kawanoue-2026}. For $\nu,m\geq0$ and integral $u$, write
\begin{equation}\label{eq:h}
	h_{\nu,m,u}(x;\mathbf{y})=
	\sum_{0\leq s_1,\dots,s_\nu\leq m}
	\frac{(-1)^{S_\nu}P(x,m+u+S_\nu)}{\binom{m+u+S_\nu+\frac12}{m+1}}
	\prod_{j=1}^{\nu}\binom{m}{s_j}
	\frac{P(y_j,u+m+S_{j-1})}{P(y_j,u+S_j)}.
\end{equation}
Obviously, $h_{0,m,u}(x;\mathbf{y})=\frac{P(x,m+u)}{\binom{m+u+\frac12}{m+1}}$. The finite sum $h_{\nu,m,u}$ is defined for all integers $u$ here and extends Kawanoue's finite sum. For some negative $u$, the sum maybe a rational function rather than a polynomial. Let
\begin{equation}\label{Hh}
	H_{\nu,m,r}=h_{\nu,m,r-m}\quad(r\geq-1),\qquad H_{m,r}=H_{1,m,r}.
\end{equation}
By lemma~\ref{lem:ratio}, we can identify \eqref{eq:regular} with
$H_{\ell-1,m,r}(x_i;\widehat{x_i})$ for $r\geq0$. Moreover, since
$R_N(t)=tP(t,N-1)$ even for $N=0$, the equation \eqref{eq:exceptional} is same as
$XH_{\ell-1,m,-1}(x_i;\widehat{x_i})$.

\vspace{1mm}

The expression $P(x,m+u)$ is not a fixed rational function of $u$ since the number of its factors depends on $u$. Therefore, the rational-function identity argument cannot be applied directly to this expression before this factor is removed. Next, we normalize the expression as
\begin{equation}\label{eq:normalized}
	\begin{split}
		\frac{h_{\nu,m,u}}{P(x,m+u)}
		=&\sum_{\mathbf s}
		\frac{(-1)^{S_\nu}}{\binom{m+u+S_\nu+\frac12}{m+1}}
		\prod_{a=1}^{S_\nu}\left(x^2-(m+u+a)^2\right)\\
		&\cdot\prod_{j=1}^{\nu}\left[\binom m{s_j}
		\prod_{a=s_j+1}^{m}\left(y_j^2-(u+S_{j-1}+a)^2\right)\right].
	\end{split}
\end{equation}
For fixed $\nu,m$, the right side is a rational function of $u,x,\mathbf y$ with fixed product lengths. Its scalar denominators do not vanish at integral $u$. A rational identity between such expressions which holds for all sufficiently large integer $u$ therefore extends to all integral $u$ by clearing denominators.

\vspace{1mm}

\begin{lem}\label{lem:rec}
	For $m\geq0$ and $u\in\mathbb{Z}$, we have
	\begin{equation}\label{eq:rec-u}
		\frac{2u-1}{2m+4}h_{1,m+1,u-1}
		=\left[x^2+y^2-(m+u+1)^2-u^2\right]h_{1,m,u}-2h_{1,m,u+1}.
	\end{equation}
	Equivalently, for $m\geq1$ and $r\geq-1$, we have
	\begin{equation}\label{eq:rec-r}
		\frac{2r-2m+1}{2m+2}H_{m,r}(x;y)
		=\left[x^2+y^2-(r+1)^2-(r-m+1)^2\right]H_{m-1,r}(x;y)
		-2H_{m-1,r+1}(x;y).
	\end{equation}
\end{lem}

\vspace{1mm}

\begin{proof}
	According to Proposition 3 in~\cite{H. Kawanoue-2026}, we can obtain the identity \eqref{eq:rec-u} for $u\geq1$. Divide it by $P(x,m+u)$, we have
	\begin{equation*}
		\begin{split}
			&\frac{2u-1}{2m+4}\frac{h_{1,m+1,u-1}}{P(x,m+u)}\\
			=&\left[x^2+y^2-(m+u+1)^2-u^2\right]\frac{h_{1,m,u}}{P(x,m+u)}
			-\frac{2h_{1,m,u+1}}{P(x,m+u+1)}\cdot\frac{P(x,m+u+1)}{P(x,m+u)}\\
			=&\left[x^2+y^2-(m+u+1)^2-u^2\right]\frac{h_{1,m,u}}{P(x,m+u)}
			-\frac{2h_{1,m,u+1}}{P(x,m+u+1)}\cdot\left[x^2-(m+u+1)^2\right].
		\end{split}	
	\end{equation*}
	Thus the normalized identity is rational in $u$ and holds at infinitely many integers; it holds identically. Specializing to any integer $u$ is legitimate. Substitute the old indices $m-1$ and $u=r-m+1$ we can get the identity~\eqref{eq:rec-r}.
\end{proof}

\vspace{1mm}

\begin{lem}\label{lem:delete}
	For integral $u$, $h_{\nu,m,u}$ is symmetric in its auxiliary variables. Consequently, for $r\geq-1$ and any $j$, we can obtain
	\begin{equation}\label{eq:delete}
		\begin{split}
			H_{\nu,m,r}(x;\mathbf y)
			=\sum_{k=0}^{m}(-1)^k\binom mk
			\prod_{q=r-m+k+1}^{r}(y_j^2-q^2)
			\cdot H_{\nu-1,m,r+k}(x;\mathbf y\setminus y_j).
		\end{split}
	\end{equation}
\end{lem}

\vspace{1mm}

\begin{proof}
	For $m=0$ the sum is independent of the auxiliary variables. For $m>0$ and positive integral $u$, the symmetry follows from Proposition 7 in~\cite{H. Kawanoue-2026}. Divide the symmetry identity by $P(x,m+u)$, and apply \eqref{eq:normalized}, the difference is a rational function of $u$ vanishing at infinitely many integers. Hence symmetry holds for every integer $u$. By grouping the terms in~\eqref{eq:h} according to $s_1=k$ and Lemma~\ref{lem:ratio}, we can verify the identity~\eqref{eq:delete} for $j=1$. With symmetry in auxiliary variables, the formula~\eqref{eq:delete} holds for every $j$.
\end{proof}

\vspace{1mm}

\begin{rema}
	This extension does not divide by the coefficient $(k+1)(2v+k+m+2)$ in Kawanoue's auxiliary recurrence. That coefficient can vanish at negative integers, for example $m=2,k=0,v=-2$. Thus the original induction cannot be applied unchanged, it would not be sufficient. Alternatively, each nonzero summand of Kawanoue's $\Phi_{m,k,v}$ is polynomial in $v$ by Lemma~\ref{lem:ratio}. Its symmetry for $v\geq0$ extends as a polynomial identity.
\end{rema}

\vspace{3mm}

\subsection{Evaluation of Finite Sums on the Boundary}

\vspace{3mm}

In this part, we evaluate the finite sum $H_{m,r}(x;y)$ for $y=-x-m$ and obtain the corresponding antisymmetry relation, which is used in the latter proofs. For convenience, write $(a)_n^\uparrow=\prod_{j=0}^{n-1}(a+j)$ and $(a)_n^\downarrow=\prod_{j=0}^{n-1}(a-j)$ for $n\geq0$. We first record the finite scalar identity needed for the boundary.
\begin{lem}[Lemma 4,~\cite{H. Kawanoue-2026}]\label{lem:sum}
	For $m\geq0$, the following holds as a rational identity in $u,k$:
	\begin{equation}\label{eq:sum}
		\sum_{s=0}^{m}
		\frac{(-1)^s\binom ms}{\binom{m+u+s+\frac12}{m+1}}
		\binom{k+m+u+s+\frac12}{2m}
		=\frac{(m+1)(k)_m^\downarrow(u+k+\frac32)_m^\uparrow}
		{(u+\frac12)_{2m+1}^\uparrow}.
	\end{equation}
\end{lem}

\vspace{1mm}

In this paper, we can also use the lemma for negative integer $u$. These values are allowed because the formula is a rational identity whose denominators do not vanish at integer $u$.

\vspace{1mm}
\begin{prop}\label{prop:boundary}
	Let $r\geq0,m\geq0$, or let $r=-1,m\geq1$. Then
	\begin{equation}\label{eq:boundary}
		H_{m,r}(x;-x-m)
		=A_{m,r}\left(x-r\right)_{m+2r+1}^{\uparrow}\left(x+\frac12\right)_m^{\uparrow},
		\qquad A_{m,r}=\frac{(m+1)(2m)!}{\left(r-m+\frac12\right)_{2m+1}^{\uparrow}}.
	\end{equation}
	In particular the restricted rational function is a polynomial and
	\begin{equation}\label{eq:boundary-sum}
		H_{m,r}(x;-x-m)+H_{m,r}(-x-m;x)=0.
	\end{equation}
\end{prop}

\vspace{1mm}

\begin{proof}
	Set $u=r-m$, $x=-k-\frac12$, and $y=-m+k+\frac12$. For $a=m+u+s$, by the definition of $P(t,a)$ and the generalized falling-factorial cancellation, we have
	\begin{equation*}
		\frac{P(-k-\frac12,a)}{P(-m+k+\frac12,a-m)}
		=-\left(k+a+\frac12\right)_{2m}^{\downarrow}.
	\end{equation*}
	From the formula~\eqref{eq:h} and Lemma~\ref{lem:sum}, the following identity can be verified,
	\begin{equation*}
		\frac{h_{1,m,u}(-k-\frac12;-m+k+\frac12)}{P(-m+k+\frac12,u+m)}
		=-\frac{(m+1)(2m)!\left(k\right)_m^{\downarrow}\left(u+k+\frac32\right)_m^{\uparrow}}{\left(u+\frac12\right)_{2m+1}^{\uparrow}}.
	\end{equation*}
	Upon replacing $k$ by $-x-\frac12$ and using the oddness of $P$, we can find
	\begin{equation}\label{eq:pre-boundary}
		h_{1,m,u}(x;-x-m)
		=\frac{(m+1)(2m)!}{\left(u+\frac12\right)_{2m+1}^{\uparrow}}
		P(x+m,u+m)\left(x+\frac12\right)_m^{\uparrow}\left(x-u-m\right)_m^{\uparrow}.
	\end{equation}
	For $r=u+m\geq0$, the consecutive products combine as
	\begin{equation*}
		P(x+m,r)\left(x-r\right)_m^{\uparrow}=\left(x-r\right)_{m+2r+1}^{\uparrow}.
	\end{equation*}
	For $r=-1$ and $m\geq1$, the exceptional cancellation is instead
	\begin{equation*}
		P(x+m,-1)\left(x+1\right)_m^{\uparrow}
		=\frac{\left(x+1\right)_m^{\uparrow}}{x+m}=\left(x+1\right)_{m-1}^{\uparrow}.
	\end{equation*}
	Thus we complete the proof of the identity~\eqref{eq:boundary}. Replace $x$ by $-x-m$, the two rising-factorial factors acquire signs $(-1)^{m+2r+1}$ and $(-1)^m$. Their product is $-1$, which proves \eqref{eq:boundary-sum}.
\end{proof}

\vspace{1mm}

\begin{rema}
	The boundary restriction to $y=-x-m$ is computed in $\mathbb{C}(x)$. The cancellation above proves that its result is polynomial. We do not use a polynomial degree argument on the unreduced rational function. For instance
	\begin{equation*}
		h_{1,1,-2}(x;y)=8x+\frac{8(y^2-1)}{3x},\qquad
		h_{1,1,-2}(x;-x-1)=\frac{32}{3}\left(x+\frac12\right).
	\end{equation*}
\end{rema}

\vspace{3mm}

\subsection{Polynomial Coefficients and Logarithmic Divisibility}

In this part, we first specify the coefficients of \eqref{eq:regular} and \eqref{eq:exceptional}, and prove that the derivations in Theorem~\ref{thm:main} belong to $D(\mathcal{C}_{\ell,m})$.

\begin{lem}\label{lem:degree}
	The coefficients in \eqref{eq:regular} and \eqref{eq:exceptional} are polynomials with degrees at most $d_r,d_*$ respectively.
\end{lem}

\vspace{1mm}

\begin{proof}
	By the definition of $p_{i,r}^{(m)}$, the product involving $y_j$ contains $m-s_j$ quadratic factors. For a regular summand, the degree bound is given by
	\begin{equation*}
		2(r+S_{\ell-1})+1+2\sum_j(m-s_j)=d_r.
	\end{equation*}
	For the exceptional summand of $p_{i,*}^{(m)}$, $\deg R_N=2N$ including $N=0$, and the bound is
	\begin{equation*}
		(\ell-1)+2S_{\ell-1}+2\sum_j(m-s_j)=d_*.
	\end{equation*}
The explicit sum-product expressions for $p_{i,r}^{(m)}$ and $p_{i,*}^{(m)}$ contain no variable denominators. Equivalently, in the $h$-notation, Lemma ~\ref{lem:ratio} shows that each quotient involving an auxiliary variable is polynomial. The only possible pole arises from $P(x_i,-1)=x_i^{-1}$ in the exceptional case and is canceled by the factor $X=x_1\cdots x_\ell$.
\end{proof}

\vspace{1mm}

\begin{prop}\label{prop:two}
For every $m\geq0$ and every $r\geq0$, we have
	\begin{equation*}
		P(x\pm y,m)\mid H_{m,r}(x;y)\pm H_{m,r}(y;x).
	\end{equation*}
Assume $G_m(x,y)=xyH_{m,-1}(x;y)$, then
	\begin{equation*}
		P(x\pm y,m)\mid G_m(x,y)\pm G_m(y,x).
	\end{equation*}
\end{prop}

\vspace{1mm}

\begin{proof}
In this proof, we consider the calculation in $\mathbb{C}[x,y]$. We first establish the divisibility assertions for all $r\geq0$, and then use them to prove the exceptional case $r=-1$.
	
\vspace{1mm}
	
\noindent\textbf{Case 1}, when $r\geq0$.
	
\vspace{1mm}
	
We prove this by induction on $m$, simultaneously for all $r\geq0$. Assume $m=0$,
	\begin{equation*}
		\begin{split}
			H_{0,r}(x;y)
			=h_{1,0,r}(x;y)
			=(r+\frac12)^{-1}\cdot P(x,r).
		\end{split}
	\end{equation*}
Then we can get $H_{0,r}(x;y)+H_{0,r}(y;x)=(r+\frac12)^{-1}\left[P(x,r)+P(y,r)\right]$. Because the function $P$ is odd, it is routine to check that $H_{0,r}(x;-x)+H_{0,r}(-x;x)=0$, which implies
	\begin{equation*}
		P(x+y,0)=x+y\mid H_{0,r}(x;y)+H_{0,r}(y;x).
	\end{equation*} 
Next assume the result holds for all smaller values of $m$, and consider the case of $m$. By~\eqref{eq:rec-r}, we have
	\begin{equation*}
		\begin{split}
			&\frac{2r-2m+1}{2m+2}\left[H_{m,r}(x;y)+H_{m,r}(y;x)\right]\\
			=&\left[x^2+y^2-(r+1)^2-(r-m+1)^2\right]\left[H_{m-1,r}(x;y)+H_{m-1,r}(y;x)\right]\\
			&-2\left[H_{m-1,r+1}(x;y)+H_{m-1,r+1}(y;x)\right].
		\end{split}
	\end{equation*}
We need to illustrate $P(x+y,m)$ divides the left-hand side. Since the induction hypothesis applies to both $r$ and $r+1$, the right-hand side is divisible by $P(x+y,m-1)$; therefore, the left-hand side is as well. By the Lemma~\ref{lem:ratio}, $P(x+y,m)=(x+y+m)(x+y-m)P(x+y,m-1)$. Thus it is sufficient to prove the left-hand side is divisible by both $x+y+m$ and $x+y-m$.
	
\vspace{1mm}
	
By proposition~\ref{prop:boundary}, the identity $H_{m,r}(x;y)+H_{m,r}(y;x)=0$ holds when $x+y+m=0$. It follows that $x+y+m$ divides the left-hand side. By~\eqref{eq:h},~\eqref{Hh} and the oddness of $P$, we have
\begin{equation}\label{Hoddeven}
	H_{m,r}(-x;y)=-H_{m,r}(x;y) \text{~and~} H_{m,r}(x;y)=H_{m,r}(x;-y).
\end{equation}
Therefore, 
\begin{equation*}
	H_{m,r}(-x;-y)+H_{m,r}(-y;-x)=
	-\left[H_{m,r}(x;y)+H_{m,r}(y;x)\right].
\end{equation*}
Consequently, if $x+y-m=0$ and by proposition~\ref{prop:boundary},
\begin{equation}\label{Hx+y-m}
	\begin{split}
		H_{m,r}(x;y)+H_{m,r}(y;x)
		=&H_{m,r}(x;m-x)+H_{m,r}(m-x;x)\\
		=&-\left[H_{m,r}(-x;x-m)+H_{m,r}(x-m;-x)\right]
		=0.
	\end{split}	
\end{equation}
This means that $x+y-m$ also divides the left-hand side and $P(x+y,m)\mid H_{m,r}(x;y)+H_{m,r}(y;x)$. Finally, replacing $y$ by $-y$ and by~\eqref{Hoddeven} , we can get $P(x-y,m)$ divides $H_{m,r}(x;y)-H_{m,r}(y;x)$, this proves both regular assertions for all $m$, $r\geq0$,
	\begin{equation*}
		P(x\pm y,m)\mid H_{m,r}(x;y)\pm H_{m,r}(y;x).
	\end{equation*}
	
\vspace{1mm}
	
\noindent\textbf{Case 2, when $r=-1$}.
	
\vspace{1mm}
	
We can use the same method to prove it. For the exceptional case  $r=-1$, we also induct on $m$. For convenience, put $T_m(x,y)=G_m(x,y)+G_m(y,x)$. When $m=0$, it is not difficult to find that $G_0(x,y)=xyH_{0,-1}(x;y)=-2y$. Thus $T_0(x,y)=G_0(x,y)+G_0(y,x)=-2(x+y)$ and hence $P(x+y,0)$ divides the sum. 

\vspace{1mm}

Assume $m\geq1$ and $P(x+y,m-1)\mid T_{m-1}(x,y)$. Taking $r=-1$ in \eqref{eq:rec-r} and multiplying by $xy$ gives the polynomial identity
\begin{equation*}
	-\frac{2m+1}{2m+2}G_m(x,y)
	=(x^2+y^2-m^2)G_{m-1}(x,y)-2xyH_{m-1,0}(x;y).
\end{equation*}
Adding its transpose yields
\begin{equation}\label{eq:exceptional-sum-rec}
	-\frac{2m+1}{2m+2}T_m(x,y)
	=(x^2+y^2-m^2)T_{m-1}(x,y)-2xy\left[H_{m-1,0}(x;y)+H_{m-1,0}(y;x)\right].
\end{equation}
The first term on the right is divisible by $P(x+y,m-1)$ by exceptional induction, and the second by the regular assertion already proved. The left scalar is nonzero, so $P(x+y,m-1)\mid T_m(x,y)$.

\vspace{1mm}

For $m\geq1$, Proposition~\ref{prop:boundary} also applies at $r=-1$ and gives an identity in $\mathbb{C}(x)$:
\begin{equation*}
	H_{m,-1}(x;-x-m)+H_{m,-1}(-x-m;x)=0.
\end{equation*}
Multiplication by $x(-x-m)$ gives $T_m(x,-x-m)=0$. Because $T_m$ is a polynomial, this is a polynomial identity in $x$; it therefore holds also at points where the original rational expressions were undefined. Hence $x+y+m\mid T_m(x,y)$. Moreover, according to~\eqref{Hoddeven}, we can get
\begin{equation*}
	G_m(-x,-y)=-G_m(x,y),\qquad T_m(-x,-y)=-T_m(x,y).
\end{equation*}
The same boundary substitution as in~\eqref{Hx+y-m} gives $x+y-m\mid T_m(x,y)$. Unique factorization now implies $P(x+y,m)\mid T_m(x,y)$ and completes exceptional induction.

\vspace{1mm}

The parities of $G_m(x,y)$ differ from those of $H_{m,r}(x;y)$ because of the prefactor $xy$:
\begin{equation*}
	G_m(x,-y)=-G_m(x,y),\qquad G_m(-y,x)=G_m(y,x).
\end{equation*}
Thus replacing $y$ by $-y$, it routine to check that
\begin{equation*}
	P(x-y,m)\mid -G_m(x,y)+G_m(y,x).
\end{equation*}
Multiplication by the unit $-1$ proves the required minus assertion. Therefore, we complete the proof of the following divisibility relation
\begin{equation*}
	P(x\pm y,m)\mid G_m(x,y)\pm G_m(y,x).
\end{equation*}
This completes the proof of the entire proposition.
\end{proof}

\vspace{1mm}

\begin{prop}\label{prop:many}
In any number $\ell\geq2$ of variables, define the regular coefficients by \eqref{eq:regular}, allowing \emph{every} $r\geq0$. Then for $b<c$, we have
	\begin{equation*}
		P(x_b\pm x_c,m)\mid p_{b,r}^{(m)}\pm p_{c,r}^{(m)}.
	\end{equation*}
Moreover, the exceptional coefficients also satisfy the analogous assertion
\begin{equation*}
	P(x_b\pm x_c,m)\mid p_{b,*}^{(m)}\pm p_{c,*}^{(m)}.
\end{equation*}
\end{prop}

\vspace{1mm}

\begin{proof}
To distinguish the dimensions, write
	\begin{equation*}
		\begin{split}
			p_{i,r}^{(m,\ell)}&=H_{\ell-1,m,r}\bigl(x_i;(x_j)_{j\ne i}\bigr)\quad(r\ge0),\\
			p_{i,*}^{(m,\ell)}&=X_\ell H_{\ell-1,m,-1}\bigl(x_i;(x_j)_{j\ne i}\bigr),
			\qquad X_\ell=\prod_{j=1}^{\ell}x_j.
		\end{split}
	\end{equation*}
All auxiliary lists are initially ordered by increasing index. These coefficients are polynomials by the same finite-product argument as in Lemma~\ref{lem:degree}, which applies to every $\ell\geq2$ and every regular parameter $r\geq0$.

\vspace{1mm}
	
We prove the regular and exceptional assertions simultaneously by induction on the number $\ell$ of variables. For $\ell=2$, the regular coefficients are $H_{m,r}(x_1;x_2)$ and $H_{m,r}(x_2;x_1)$, while the exceptional coefficients are $G_m(x_1,x_2)$ and $G_m(x_2,x_1)$. In this case, the required divisibility relations follow from Proposition~\ref{prop:two}.

\vspace{1mm}
	
Let $\ell\geq3$ and assume the assertion for $\ell-1$ variables, simultaneously for all regular parameters $t\ge0$ and the exceptional family. Fix $b<c$ and $\varepsilon\in\{1,-1\}$, and choose $d\notin\{b,c\}$. Set
\begin{equation*}
	R_{\widehat d}=\mathbb{C}[x_j:j\ne d],\qquad
	X_{\widehat d}=\prod_{j\ne d}x_j,\qquad
	L=P(x_b+\varepsilon x_c,m).
\end{equation*}
Then $L\in R_{\widehat d}$, $\mathbb{C}[x_1,\ldots,x_\ell]=R_{\widehat d}[x_d]$, and $X_\ell=x_dX_{\widehat d}$. For $i\ne d$, define
\begin{equation*}
	q_{i,t}=H_{\ell-2,m,t}\bigl(x_i;(x_j)_{j\ne i,d}\bigr)\quad(t\ge0),
	\qquad
	q_{i,*}=X_{\widehat d}H_{\ell-2,m,-1}\bigl(x_i;(x_j)_{j\ne i,d}\bigr).
\end{equation*}
These are precisely the ordinary and exceptional coefficients for the remaining $\ell-1$ variables. Relabeling those variables if necessary, by the induction hypothesis, there exist some polynomials $U_t,U_*\in R_{\widehat d}$ such that
\begin{equation}\label{eq:lower-quotients}
	q_{b,t}+\varepsilon q_{c,t}=LU_t\quad(t\ge0),\qquad
	q_{b,*}+\varepsilon q_{c,*}=LU_*.
\end{equation}

\vspace{1mm}

For $r\geq-1$ and $0\leq k\leq m$, we define
\begin{equation}\label{eq:deletion-multiplier}
	A_{r,k}(x_d)=(-1)^k\binom mk
	\prod_{a=r-m+k+1}^{r}(x_d^2-a^2).
\end{equation}
The product has $m-k$ factors, so $A_{r,k}\in\mathbb{C}[x_d]$, including $k=m$, when the product is empty. By the symmetry of the auxiliary variables established in Lemma~\ref{lem:delete}, 
$x_d$ may be moved to the first auxiliary position in either coefficient. Equation~\eqref{eq:delete} consequently yields
\begin{equation}\label{eq:regular-deletion}
	p_{i,r}^{(m,\ell)}=\sum_{k=0}^{m}A_{r,k}(x_d)q_{i,r+k}
	\quad(i=b,c;\ r\ge0).
\end{equation}
The multiplier is identical for $i=b$ and $i=c$: it depends on $m,r,k,x_d$, but not on the principal variable. Thus it follows from this expression that
\begin{equation*}
	\begin{split}
		p_{b,r}^{(m,\ell)}+\varepsilon p_{c,r}^{(m,n)}
		&=\sum_{k=0}^{m}A_{r,k}(x_d)
		\bigl(q_{b,r+k}+\varepsilon q_{c,r+k}\bigr)\\
		&=L\sum_{k=0}^{m}A_{r,k}(x_d)U_{r+k}.
	\end{split}
	\label{eq:regular-factorization}
\end{equation*}
The second equality holds by \eqref{eq:lower-quotients}, which is applicable because $r+k\ge0$. The last sum belongs to $R_{\widehat d}[x_d]$, so the desired regular divisibility follows. In particular, it is essential that the induction hypothesis holds for every nonnegative parameter: $r+k$ need not lie in the range eventually selected for the basis.

\vspace{1mm}

For $i=b,c$ and $r=-1$, multiply the expression~\eqref{eq:delete} by $X_\ell$:
\begin{equation*}
	p_{i,*}^{(m,\ell)}=X_\ell\sum_{k=0}^{m}A_{-1,k}(x_d)
	H_{\ell-2,m,k-1}\bigl(x_i;(x_j)_{j\ne i,d}\bigr),
\end{equation*}
where
\begin{equation*}
	A_{-1,k}(x_d)=(-1)^k\binom mk\prod_{a=-m+k}^{-1}(x_d^2-a^2).
\end{equation*}
For $k=0$, the parameter remains $-1$, and $X_\ell=x_dX_{\widehat d}$ implies
\begin{equation*}
	X_\ell A_{-1,0}(x_d)H_{\ell-2,m,-1}
	=x_dA_{-1,0}(x_d)q_{i,*},
\end{equation*}
where the arguments of $H$ are the same as in the preceding display. This expression is polynomial because $q_{i,*}$ is the lower-dimensional exceptional coefficient. For $k\ge1$, the parameter $k-1$ is nonnegative, and the $H$ term equals the ordinary coefficient $q_{i,k-1}$. We have therefore obtained the polynomial identity
\begin{equation}\label{eq:exceptional-deletion}
	p_{i,*}^{(m,\ell)}=x_dA_{-1,0}(x_d)q_{i,*}
	+X_\ell\sum_{k=1}^{m}A_{-1,k}(x_d)q_{i,k-1}.
\end{equation}
Then we have
\begin{equation*}
	\begin{split}
		p_{b,*}^{(m,\ell)}+\varepsilon p_{c,*}^{(m,\ell)}
		=x_dA_{-1,0}(x_d)(q_{b,*}+\varepsilon q_{c,*})
		+X_\ell\sum_{k=1}^{m}A_{-1,k}(x_d)
		(q_{b,k-1}+\varepsilon q_{c,k-1}).
	\end{split}
\end{equation*}
Applying respectively the exceptional and regular parts of \eqref{eq:lower-quotients}, we obtain
\begin{equation}\label{eq:exceptional-factorization}
	\begin{split}
		p_{b,*}^{(m,\ell)}+\varepsilon p_{c,*}^{(m,\ell)}
		=L\left[x_dA_{-1,0}(x_d)U_*+
		X_\ell\sum_{k=1}^{m}A_{-1,k}(x_d)U_{k-1}\right].
	\end{split}
\end{equation}
The bracketed expression is a polynomial in all $\ell$ variables. This proves the exceptional divisibility.

\vspace{1mm}

If $m=0$, the sums over $1\le k\le m$ in the last two steps are empty and equal zero, while $A_{-1,0}=1$; the argument remains valid. Both signs have been treated by the arbitrary choice of $\varepsilon$, and the pair $b<c$ was arbitrary. The induction on $\ell$ is complete.
\end{proof}

\vspace{1mm}

\begin{coro}\label{cor:log}
The derivations in Theorem~\ref{thm:main} belong to $D(\mathcal{C}_{\ell,m})$.
\end{coro}

\vspace{1mm}

\begin{proof}
The Euler derivation is logarithmic for any central arrangement, whereas the remaining derivations annihilate $z$. Combining Proposition~\ref{prop:many} with homogenization to the degrees in Lemma~\ref{lem:degree}, we obtain that $x_b\pm x_c-kz$ divides $f_{b,a}\pm f_{c,a}$. To see this, note that homogenizing a polynomial divisible by $P(x_b\pm x_c,m)$ preserves divisibility by the homogeneous product of $2m+1$ linear factors. This establishes the divisibility for every defining hyperplane.
\end{proof}

\vspace{3mm}

\section{Proof of the Main Theorem}

\vspace{3mm}

This section we proves the main theorem~\ref{thm:main}. For convenience, we first define lexicographic order $x_1>x_2>\cdots>x_\ell$ on the dehomogenized coefficient polynomials. 

\vspace{1mm}

\begin{lem}\label{lem:LT}
	For $1\leq i\leq\ell$ and $r\geq0$, the leading terms in the dehomogenized coefficient polynomials $p_{i,r}^{(m)}$ and $p_{i,*}^{(m)}$ are
	\begin{equation}\label{eq:LTreg}
		\operatorname{LT}(p_{i,r}^{(m)})
		=\frac{(-1)^{m(\ell-i)}}{\binom{r+m(\ell-i)+\frac12}{m+1}}
		\left(\prod_{j<i}x_j^{2m}\right)x_i^{2m(\ell-i)+2r+1},
	\end{equation}
	\begin{equation}\label{eq:LTstar}
		\operatorname{LT}(p_{i,*}^{(m)})
		=\frac{(-1)^{m(\ell-i)}}{\binom{m(\ell-i)-\frac12}{m+1}}
		\left(\prod_{j<i}x_j^{2m+1}\right)x_i^{2m(\ell-i)}
		\left(\prod_{j>i}x_j\right).
	\end{equation}
\end{lem}

\vspace{1mm}

\begin{proof}
We use the notations $\mathbf{y}=(y_1,y_2,\dots,y_{\ell-1})
=(x_1,x_2,\dots,\widehat{x_i},\dots,x_\ell)$, $s_i$ and $S_i$ for $1\leq i\leq\ell$ in the coefficient polynomials. For fixed $i$, the summation index $t_j$ of $x_j$ is defined by
	\begin{equation*}
		t_j=
		\begin{cases}
			s_j,\quad &j<i,\\
			s_{j-1},&j>i,
		\end{cases}
	\end{equation*}
	so $0\leq t_j\leq m$. In terms of $t_j$, the exponent of $x_j$ in a summand of $p_{i,r}^{(m)}$ is
	\begin{equation*}
		\begin{cases}
			2(m-t_j),\quad &j\ne i,\\
			2r+2S_{\ell-1}+1,\quad &j=i.\\
		\end{cases}
	\end{equation*}
	To obtain the leading term under the lexicographic order $x_1>x_2 > \dots > x_\ell$, we first maximize the exponent of $x_1$, then that of $x_2$, and so on. Hence we must take $t_j=0$ for $j<i$, and then maximize the exponent of $x_i$, which forces $t_j=m$ for $j>i$. Therefore there is exactly one summand contributing to the leading term, and its leading monomial is
	\begin{equation*}
		\left(\prod_{j<i}x_j^{2m}\right)
		x_i^{2m(\ell-i)+2r+1}.
	\end{equation*}
	Reading off the corresponding scalar coefficient gives
	\begin{equation*}
		\operatorname{LT}(p_{i,r}^{(m)})
		=\frac{(-1)^{m(\ell-i)}}{\binom{r+m(\ell-i)+\frac12}{m+1}}
		\left(\prod_{j<i}x_j^{2m}\right)
		x_i^{2m(\ell-i)+2r+1}.
	\end{equation*}
	By the same argument, using the definition of $R_{S_{\ell-1}}$ and the extra factor $X$, we obtain
	\begin{equation*}
		\operatorname{LT}(p_{i,*}^{(m)})
		=\frac{(-1)^{m(\ell-i)}}{\binom{m(\ell-i)-\frac12}{m+1}}
		\left(\prod_{j<i}x_j^{2m+1}\right)
		x_i^{2m(\ell-i)}
		\left(\prod_{j>i}x_j\right).
	\end{equation*}
\end{proof}

\vspace{1mm}

Let $B=(p_{i,a}^{(m)})$ have rows $i=1,\dots,\ell$ and columns $a=0,1,\dots,\ell-2,*$. We next establish the following properties of determinant of $B$.

\vspace{1mm}

\begin{prop}\label{prop:det}
	The determinant of $B$ is nonzero. Its leading term is
	\begin{equation}\label{eq:detLT}
		\operatorname{LT}(\det B)=C_{\ell,m}\prod_{i=1}^{\ell}x_i^{2(2m+1)(\ell-i)},
	\end{equation}
	where
	\begin{equation*}
		C_{\ell,m}=
		\frac{\displaystyle(-1)^{m\frac{\ell(\ell-1)}{2}+\frac{(\ell-1)(\ell-2)}{2}}}
		{\displaystyle\binom{-\frac12}{m+1}
			\prod_{n=1}^{\ell-1}\binom{(m+1)n-\frac12}{m+1}}
		\ne0.
	\end{equation*}
\end{prop}

\vspace{1mm}

\begin{proof}
	For each permutation term in the expansion of the determinant of $B$, its leading monomial is the product of the leading monomials of its entries. We show that a unique permutation maximizes that product.
	
	Suppose the first $i-1$ rows have been assigned their maximizing columns, and put $n=\ell-i+1$. The remaining regular columns are $r=0,\cdots,n-2$, together with the exceptional column. The already selected entries are regular and their leading monomials contain no variables later than their row index, so they contribute nothing to the exponent of $x_i$.
	
	If row $i$ receives a regular column $r$, the total exponent of $x_i$ among all remaining leading entries is
	\begin{equation*}
		[2m(n-1)+2r+1]+(n-2)2m+(2m+1)
		=4m(n-1)+2r+2.
	\end{equation*}
	If row $i$ receives the exceptional column, that exponent is instead
	\begin{equation*}
		2m(n-1)+(n-1)2m=4m(n-1).
	\end{equation*}
	Thus the unique maximizing choice is the largest remaining regular parameter $r=n-2$, as long as $n\geq2$. The last row receives the exceptional column. The unique maximizing assignment is therefore
	\begin{equation*}
		(\ell-2,\ell-3,\dots,0,*)
	\end{equation*}
	from top row to bottom row. Its exponent of $x_i$ is $2(2m+1)(\ell-i)$.
	
	Reversing the first $\ell-1$ columns has sign $(-1)^{(\ell-1)(\ell-2)/2}$. For row $i<\ell$, set $n'=\ell-i$; its scalar denominator is
	\begin{equation*}
		\binom{(n'-1)+mn'+\frac12}{m+1}
		=\binom{(m+1)n'-\frac12}{m+1},
	\end{equation*}
	and its sign is $(-1)^{mn'}$. The last entry contributes $\binom{-1/2}{m+1}^{-1}$. Multiplying these factors proves \eqref{eq:detLT} and \eqref{eq:constant}. Since the maximizing permutation is unique, no other determinant term cancels it.
\end{proof}

\vspace{1mm}

\begin{proof}[\textbf{Proof of Theorem~\ref{thm:main}}]
It follows Corollary~\ref{cor:log} that $\theta_E,\theta_0,\theta_1,\dots,\theta_{\ell-2},\theta_*\in D(\mathcal{C}_{\ell,m})$. Let
\begin{equation*}
	\Delta=\operatorname{det}[\theta_E,\theta_0,\theta_1,\dots,\theta_{\ell-2},\theta_*]
\end{equation*}
denote the determinant of the full homogeneous coefficient matrix. Then $\Delta=z\det(f_{i,a})$. At $z=1$, its remaining factor specializes to $\det B\ne0$, so $\Delta\ne0$. Let $\alpha_H$ be the defining linear form of the hyperplane $H$ in $\mathcal{C}_{\ell,m}$, and choose linear coordinates with $\alpha_H$ as the first coordinate. Since each column of the coefficient matrix represents a logarithmic derivation, every entry in the first row is divisible by $\alpha_H$. Thus it is not difficult to find that $\alpha_H$ divides the determinant in these coordinates. As a linear change of coordinates multiplies the determinant by a nonzero scalar, it follows that $\alpha_H\mid\Delta$. This holds for every defining linear form. Since these forms are pairwise nonassociate, their product $Q_{\ell,m}$ divides $\Delta$.

\vspace{1mm}

From Lemma~\ref{lem:LT}, we deduce that the upper bound on the degrees of these derivations $\theta_E,\theta_0,\theta_1,\dots,\theta_{\ell-2},\theta_*$, as stated in Lemma~\ref{lem:degree}, is attainable. The degrees satisfy
	\begin{equation*}
		1+\sum_{r=0}^{\ell-2}d_r+d_*
		=1+(2m+1)\ell(\ell-1)=\deg Q_{\ell,m}.
	\end{equation*}
	Hence $\Delta=cQ_{\ell,m}$ for some nonzero scalar $c$. At $z=1$, the leading monomial of $Q_{\ell,m}(x,1)$ has coefficient $1$ and is the monomial in \eqref{eq:detLT}. Thus $c=C_{\ell,m}$. By Saito's determinant criterion \cite{K. Saito-1980}, we can prove that the derivations form a basis.
\end{proof}

\vspace{3mm}

\section{Restriction at infinity and the integral basis}

\vspace{3mm}

Define formal integration of even Laurent monomials by
\begin{equation*}
	\int^{x_i}t^{2a}\,dt=\frac{x_i^{2a+1}}{2a+1}\qquad(a\in\mathbb{Z}).
\end{equation*}
There is no $t^{-1}$ term, so this is well defined. For negative powers it is not a definite integral from $0$. Put
\begin{equation*}
	I_{i,u}^{(m)}=\int^{x_i}t^{2u}\prod_{j=1}^{\ell}(t^2-x_j^2)^m\,dt,
	\qquad \eta_u^{(m)}=\sum_i I_{i,u}^{(m)}\partial_{x_i}.
\end{equation*}
Feigin-Wang-Yoshinaga (Example 3.9(1), \cite{M. Feigin-2025}) prove that
\begin{equation}\label{eq:FWY}
	\eta_{-m}^{(m)},\eta_{-m+1}^{(m)},\cdots,\eta_{-m+\ell-2}^{(m)},
	\quad X\eta_{-m-1}^{(m)}
\end{equation}
form a basis of $D(D_\ell,2m+1)$. The following establishes the exact normalization of our lift.
\begin{prop}\label{prop:top}
	With $c_{\ell,m}=2(m+1)(-1)^{m\ell}$,
	\begin{equation*}
		\theta_r|_{z=0}=c_{\ell,m}\eta_{r-m}^{(m)},\qquad
		\theta_*|_{z=0}=c_{\ell,m}X\eta_{-m-1}^{(m)}.
	\end{equation*}
\end{prop}
\begin{proof}
	For any integral $a$, expansion and partial fractions give
	\begin{equation}\label{eq:formal}
		\int^{x}t^{2a}(t^2-x^2)^m\,dt
		=\frac{(-1)^m x^{2a+2m+1}}
		{2(m+1)\binom{a+m+\frac12}{m+1}}.
	\end{equation}
	Indeed the coefficient on the left is
	\begin{equation*}
		\sum_{b=0}^{m}\frac{(-1)^{m-b}\binom mb}{2a+2b+1}
		=\frac{(-1)^m m!}{2\prod_{b=0}^{m}(a+b+\frac12)}.
	\end{equation*}
	The last equality follows, for example, by comparing the simple-pole residues of the rational functions in $a$ and their values at infinity. All denominators are nonzero at integral $a$.
	
	Expand the $\ell-1$ factors involving $y_j$ in $I_{i,u}^{(m)}$ and apply \eqref{eq:formal} with $a=u+S_{\ell-1}$. This yields
	\begin{equation}\label{eq:I-expansion}
		I_{i,u}^{(m)}=
		\sum_{\mathbf s}
		\frac{(-1)^{m\ell-S_{\ell-1}}x_i^{2(u+m+S_{\ell-1})+1}}
		{2(m+1)\binom{u+m+S_{\ell-1}+\frac12}{m+1}}
		\prod_j\binom m{s_j}y_j^{2(m-s_j)}.
	\end{equation}
	For $u=r-m$ and $r\geq0$, multiplying this by $c_{\ell,m}$ gives precisely the degree-$d_r$ component of \eqref{eq:regular}.
	
	For $u=-m-1$, the smallest power of $x_i$ in \eqref{eq:I-expansion} is $-1$. Multiply by $X$ before taking polynomial homogeneous components. The degree-$d_*$ component of \eqref{eq:exceptional} has $x_i^{2S_{\ell-1}}\prod_{j\ne i}x_j$ in the corresponding term, so it equals $c_{\ell,m}XI_{i,-m-1}^{(m)}$. This argument does not apply a polynomial-only top-degree lemma to a rational function. Finally, setting $z=0$ in the homogenizations retains exactly these top components.
\end{proof}
This identifies the geometric motivation for the parameter choices and gives a second proof of independence using \eqref{eq:FWY}. 

\vspace{3mm}

\section{Examples}

\vspace{3mm}

In this part, we consider the case $\mathcal{C}_{4,1}=c\operatorname{Cat}(D_4,1)$. Its defining polynomial is
\begin{equation*}
	Q_{4,1}=
	z\prod_{1\le i<j\le4}
	(x_i-x_j)\bigl((x_i-x_j)^2-z^2\bigr)
	(x_i+x_j)\bigl((x_i+x_j)^2-z^2\bigr).
\end{equation*}
We give the coefficients of the basis explicitly, retaining the normalization of the general construction. For $i\in\{1,2,3,4\}$, write $t_i=x_i^2,w=z^2$, and set
\begin{equation*}
e_{1,i}=\sum_{j\ne i}x_j^2,\qquad
e_{2,i}=\sum_{\substack{j<k\\j,k\ne i}}x_j^2x_k^2,\qquad
e_{3,i}=\prod_{j\ne i}x_j^2.
\end{equation*}

\vspace{1mm}

The homogenized regular coefficients are
\begin{equation*}
	\begin{split}
f_{i,0}=&-\frac{8x_i}{105}\Bigl[3t_i^3-7wt_i^2+7w^2t_i-3w^3-7(t_i^2-w^2)e_{1,i}+35(t_i-w)e_{2,i}+105e_{3,i}\Bigr],\\[12pt]
f_{i,1}=&-\frac{8x_i(t_i-w)}{315}\Bigl[5t_i^3-19wt_i^2+23w^2t_i-3w^3-(9t_i^2-12wt_i+9w^2)e_{1,i}+21(t_i+w)e_{2,i}-105e_{3,i}\Bigr],\\[12pt]
f_{i,2}=&-\frac{8x_i(t_i-w)(t_i-4w)}{3465}\Bigl[35t_i^3-155wt_i^2+207w^2t_i-39w^3-(55t_i^2-88wt_i+33w^2)e_{1,i}\\
&+(99t_i+33w)e_{2,i}-231e_{3,i}\Bigr],\\[12pt]
f_{i,*}=&-\frac8{15}\prod_{j\ne i}x_j\Bigl[t_i^3+5wt_i^2-21w^2t_i+5w^3-(5t_i^2-20wt_i+5w^2)e_{1,i}-(15t_i-5w)e_{2,i}-5e_{3,i}\Bigr].
	\end{split}
\end{equation*}
These formulas give all sixteen coefficients as $i$ ranges from $1$ to $4$. Their dehomogenizations are precisely
\begin{equation*}
	p_{i,r}^{(1)}(x)=f_{i,r}(x,1)
	\quad (r=0,1,2),
	\qquad
	p_{i,*}^{(1)}(x)=f_{i,*}(x,1).
\end{equation*}
Equivalently,
\begin{equation*}
	f_{i,r}(x,z)=z^{7+2r}p_{i,r}^{(1)}(x/z),
	\qquad
	f_{i,*}(x,z)=z^9p_{i,*}^{(1)}(x/z).
\end{equation*}
The five derivations are therefore
\begin{equation*}
	\theta_E,\qquad\theta_r=\sum_{i=1}^{4}f_{i,r}\partial_{x_i}\text{~for~}r=0,1,2,\qquad\theta_*\sum_{i=1}^{4}f_{i,*}\partial_{x_i}.
\end{equation*}
They are homogeneous of degrees $1,7,9,11,9$, respectively.

\vspace{1mm}

For completeness, logarithmicity can be checked directly from the coefficient formulas. For every $a\in\{0,1,2,*\}$, $i<j$, $\varepsilon\in\{-1,1\}$, and $k\in\{-1,0,1\}$, substitution gives $\left.(f_{i,a}+\varepsilon f_{j,a})
\right|_{x_i=kz-\varepsilon x_j}=0$. Thus each defining form $x_i+\varepsilon x_j-kz$ divides
$f_{i,a}+\varepsilon f_{j,a}$. Together with $\theta_a(z)=0$, this proves that the four derivations belong to $D(\mathcal C_{4,1})$. In the coordinate order $(z,x_1,x_2,x_3,x_4)$, their coefficient determinant is
\begin{equation*}
\Delta=\operatorname{det}
\begin{pmatrix}
	z&0&0&0&0\\
	x_1&f_{1,0}&f_{1,1}&f_{1,2}&f_{1,*}\\
	x_2&f_{2,0}&f_{2,1}&f_{2,2}&f_{2,*}\\
	x_3&f_{3,0}&f_{3,1}&f_{3,2}&f_{3,*}\\
	x_4&f_{4,0}&f_{4,1}&f_{4,2}&f_{4,*}
\end{pmatrix}.
\end{equation*}
Since all five derivations are logarithmic, $Q_{4,1}\mid\Delta$. The equality $1+7+9+11+9=37=\deg Q_{4,1}$ implies that $\Delta$ is a constant multiple of $Q_{4,1}$. Evaluating at $(z,x_1,x_2,x_3,x_4)=(1,2,5,9,14)$ determines this constant and gives
\begin{equation*}
	\Delta=-\frac{4096}{31185}\,Q_{4,1}.
\end{equation*}
Consequently, by Saito’s criterion we get the homogeneous basis $\theta_E,\theta_0,\theta_1,\theta_2,\theta_*$ of $D(\mathcal C_{4,1})$, with exponents $\exp(\mathcal C_{4,1})=(1,7,9,9,11)$.

\vspace{3mm}

\section*{Acknowledgments}

\vspace{3mm}

The first author is supported by the Natural Science Funds of China (12201029).


\vspace{3mm}

\begin{thebibliography}{99}
	
\vspace{3mm}
	
\bibitem{P. H. Edelman-1996} P. H. Edelman, V. Reiner, Free arrangements and rhombic tilings, Discrete and Computational Geometry, 15 (1996), 307-340.
	
\bibitem{M. Feigin-2025} M. Feigin, Z. Wang, M. Yoshinaga, Integral expressions for derivations of multiarrangements, International Journal of Mathematics, 36 (2025), 2450085.
	
\bibitem{R. Gao-2026} R. Gao, Y. Si, The Bernoulli basis construction for the extended Shi and Catalan arrangements of the type $A_{\ell-1}$, Communications in Algebra, (2026), 1-14.
	
\bibitem{H. Kawanoue-2026} H. Kawanoue, On the extended Catalan arrangement of type $B_{\ell}$, Proceedings of the American Mathematical Society, 154 (2026), 1197-1206.
	
\bibitem{P. Orlik-1992} P. Orlik, H. Terao, Arrangements of hyperplanes, Grundlehren der Mathematischen Wissenschaften, vol. 300, Springer, Berlin, 1992.
	
\bibitem{A. Postnikov-2000} A. Postnikov, R. P. Stanley, Deformations of Coxeter hyperplane arrangements, Journal of Combinatorial Theory, Series A, 91 (2000), 544-597.
	
\bibitem{K. Saito-1980} K. Saito, Theory of logarithmic differential forms and logarithmic vector fields, Journal of the Faculty of Science, the University of Tokyo. Section 1A, Mathematics, 27 (1980), 265-291.
	
\bibitem{D. Suyama-2021} D. Suyama, M. Yoshinaga, The primitive derivation and discrete integrals, SIGMA, 17 (2021), 038.	
	
\bibitem{M. Yoshinaga-2004} M. Yoshinaga, Characterization of a free arrangement and conjecture of Edelman and Reiner, Inventiones Mathematicae, 157 (2004), 449-454.
	
\bibitem{G. Ziegler-1989} G. Ziegler, Multiarrangements of hyperplanes and their freeness, Singularities (Iowa City, IA, 1986), 90 (1989), 345-359.		
	
\vspace{3mm}
	
\end{thebibliography}
\end{document}